\documentclass{scrartcl}

\usepackage{todonotes,xcolor}

\usepackage{amsmath}
\usepackage{amsthm}
\usepackage{amssymb}
\usepackage{amsfonts}
\usepackage{mathrsfs}
\usepackage{color}
\usepackage[utf8]{inputenc}
\usepackage{enumitem}
\usepackage{multicol}
\usepackage[english]{babel}
\usepackage[babel]{csquotes}
\usepackage{mathtools}
\usepackage{rotating}
\usepackage{verbatim}
\usepackage{thmtools}
\usepackage{makecell,booktabs,amsxtra}

\usepackage{adjustbox}
\usepackage{ifthen}
\usepackage[position=top]{subfig}
\usepackage{emptypage}
\usepackage{pgfplots}
\pgfplotsset{xticklabels={}, yticklabels={} width=7cm,compat=1.8}
\usepackage{pgfplotstable}
\pgfmathsetseed{1138} 
\pgfplotstableset{ 
	create on use/x/.style={create col/expr={42+2*\pgfplotstablerow}},
	create on use/y/.style={create col/expr={(0.6*\thisrow{x}+130)+5*rand}}
}
\pgfplotstablenew[columns={x,y}]{30}\loadedtable

\usepackage[ruled]{algorithm2e}

\SetAlFnt{\small}
\SetAlCapFnt{\small}
\SetAlCapNameFnt{\small}
\SetAlCapHSkip{0pt}
\IncMargin{-\parindent}

\usepackage{url}
\usepackage[square,numbers,sort]{natbib}

\usepackage{pgfplots,pgfplotstable}
\pgfplotsset{compat=1.8}

\usepackage[bookmarks=false,colorlinks=true,linkcolor=blue,urlcolor=blue,citecolor=red,pdfborder={0 0 0.5}]{hyperref}

\theoremstyle{plain}
\newtheorem{theorem}{Theorem}
\newtheorem{lemma}[theorem]{Lemma}

\newtheorem{corollary}[theorem]{Corollary}
\newtheorem{conjecture}[theorem]{Conjecture}

\declaretheoremstyle[
notefont=\bfseries, notebraces={}{},
bodyfont=\normalfont\itshape,
headformat=\NAME \NOTE
]{nopar}
\declaretheorem[style=nopar,name=Corollary]{corollary*}
\declaretheorem[style=nopar,name=Theorem]{theorem*}

\theoremstyle{definition}
\newtheorem{definition}[theorem]{Definition}

\newtheorem{observation}[theorem]{Observation}

\newtheorem*{acknowledgements*}{Acknowledgements}
\newtheorem*{DeclarationonGenerativeAI*}{Declaration on Generative AI}

\newcount \mycount
\newcount \mycountb

\DeclareMathOperator{\csp}{CSP}

\DeclareMathOperator{\A}{{\mathbb A}}
\DeclareMathOperator{\B}{{\mathbb B}}
\DeclareMathOperator{\C}{{\mathbb C}}

\DeclareMathOperator{\F}{{\mathbb F}}

\DeclareMathOperator{\stCon}{st-Con}
\newcommand{\lin}[1]{\operatorname{3Lin}_{#1}}

\DeclareMathOperator{\wt}{weight}
\newcommand{\real}[2]{\operatorname{real}^{#2}(#1)}
\newcommand{\zzreal}[2]{\operatorname{zz-real}^{#2}(#1)}
\DeclareMathOperator{\extremes}{extremes}
\DeclareMathOperator{\last}{last}
\DeclareMathOperator{\first}{first}
\DeclareMathOperator{\trace}{trace}
\DeclareMathOperator{\struc}{struc}
\DeclareMathOperator{\expansion}{expansion}
\DeclareMathOperator{\maxArity}{maxk}
\renewcommand{\lneq}{<}
\newcommand{\m}{|}

\providecommand{\dotdiv}{
  \mathbin{
    \vphantom{+}
    \text{
      \mathsurround=0pt 
      \protect\ooalign{
        \noalign{\kern-.45ex}
        \hidewidth$\smash{\cdot}$\hidewidth\cr 
        \noalign{\kern.45ex}
        $-$\cr 
      }%
    }%
  }%
}
\providecommand{\cupdot}{
  \mathbin{
    \vphantom{+}
    \text{
      \mathsurround=0pt 
      \protect\ooalign{
        \noalign{\kern-.4ex}
        \hidewidth$\smash{\cdot}$\hidewidth\cr 
        \noalign{\kern.4ex}
        $\cup$\cr 
      }%
    }%
  }%
}

\theoremstyle{definition}

\theoremstyle{definition}

\makeatletter
\newcommand\footnoteref[1]{\protected@xdef\@thefnmark{\ref{#1}}\@footnotemark}
\makeatother

\makeatletter
\newcommand*{\rom}[1]{\expandafter\@slowromancap\romannumeral #1@}
\makeatother

\makeatletter
\def\blfootnote{\xdef\@thefnmark{}\@footnotetext}
\makeatother

\usepackage{setspace}

\usepackage{epigraph}
\DeclareMathOperator{\HornSAT}{Horn-3SAT}

\usepackage{siunitx}
\newcommand{\separatingPotato}{
\begin{tikzpicture}[scale=0.75]
  \node (X) at (0,2.60) {$x$};
  \node (Z) at (6.2,2.60) {$z$};
  \draw (X) -- node[above] {$\omega$} (Z);

  \draw (0,0)      ellipse [x radius=0.80cm,y radius=1.75cm];
  \draw (6.2,0) ellipse [x radius=0.80cm,y radius=1.75cm];
  \node at (0,-2.08)   {$L_{r_n}(x)$};
  \node at (6.2,-2.08) {$L_{r_n}(z)$};

  \draw (0,0.90)  ellipse [x radius=0.34cm,y radius=0.42cm];
  \draw (6.2,0.50) ellipse [x radius=0.5cm,y radius=1.1cm];
  \node (Lx) at (1.60,1.60) {$L^a_{\maxArity}(x)$};
  \draw[->] (0.7,1.35) -- (0.3,1.08);
  \node at (8.2,1.35) {$L^a_{\maxArity}(z)$};
  \draw[->] (7.2,1.35) -- (6.62,1.08);

  \node (a) at (0,0.90)  {$a$};
  \node (b) at (0,-0.55) {$b$};
  \node[] (c) at (6.29,0.55) {$c$};

  \draw (0.2,0.9) -- (1.20,0.90)
        -- (1.00,0.75) -- (1.80,0.75) 
        -- (1.50,0.60) -- (2.40,0.60) 
        -- (1.80,0.430) -- (2.20,0.430)
        -- (2.00,0.30) -- (3.50,0.30) 
        -- (3.00,0.6) -- (4.20,0.60) 
        -- (3.90,0.43) -- (4.80,0.43)
        -- (4.35,0.55) -- (c);

  \draw (b) -- (c);

\end{tikzpicture}
}

\newcommand{\walkExtension}{
\begin{tikzpicture}[scale=0.75]
  \node (X) at (0,2.60) {$x$};
  \node (Z) at (6.2,2.60) {$y_i$};
  \draw (X) -- node[above] {$\omega$} (Z);
  \node (Y) at (9.2,2.60) {$y_o$};
  \draw (Z) -- node[above] {$\pi_{i,o}(R)$} (Y);

  \draw (0,0)      ellipse [x radius=0.80cm,y radius=1.75cm];
  \draw (6.2,0) ellipse [x radius=0.80cm,y radius=1.75cm];
  \draw (9.2,0) ellipse [x radius=0.80cm,y radius=1.75cm];
  \node at (0,-2.2)   {$L(x)$};
  \node at (6.2,-2.2) {$L(y_i)$};
  \node at (9.2,-2.2) {$L(y_o)$};

  \node (a) at (0,0.90)  {$a$};
  \node (b) at (0,-0.55) {$b$};

  \node (si) at (6.2,0.90)  {$s_i$};
  \node (ti) at (6.2,-0.55) {$t_i$};

  \node (so) at (9.2,0.90)  {$s_o$};
  \node (ro) at (9.2,0.1825)  {$r_o$};
  \node (to) at (9.2,-0.55) {$t_o$}; 
  
  \draw[very thick] (0.2,0.9) 
        -- (1.00,0.75) -- (1.80,0.75) 
        -- (1.50,0.60) -- (2.40,0.60) 
        -- (1.80,0.430) -- (2.20,0.430)
        -- (2.00,0.30) -- (3.50,0.30) 
        -- (3.00,0.6) -- (4.20,0.60) 
        -- (3.90,0.43) -- (4.80,0.43)
        -- (4.5,0.55) -- (8.80,0.55)
        -- (6.2,0.3) -- (ro)
        ;
  \draw[very thick] (0.2,-0.55) 
        -- (1.00,-0.25) -- (2.80,-0.25) 
        -- (1.20,-0.0) -- (4.20,0.00) 
        -- (3.20,-0.2) -- (5.80,-0.2)
        -- (4.5,0.0) -- (6.20,0.0)
         -- (ro)
        ;

  \draw (a) -- (si);
  \draw[very thick] (b) -- (ti);
  \draw (so) -- (si);
  \draw[very thick] (to) -- (ti);

\end{tikzpicture}
}

\newcommand{\abcd}{
\begin{tikzpicture}[scale=0.75]
  \node (X) at (0,1.90) {$x$};
  \node (Z) at (6.2,1.90) {$y$};
  \draw (X) -- node[above] {$\omega$} (Z);

  \node (a) at (0,0.90)  {$a$};
  \node (b) at (0,-0.85) {$b$};
  
  \node (si) at (3,0.90)  {$s_i$};
  \node (ri) at (3,-0.0) {$r_i$};
  \node (ti) at (3,-0.85) {$t_i$};
  
  \node (si2) at (4.6,0.90)  {$s_{i+1}$};
  \node (ri2) at (4.6,-0.2) {$r_{i+1}$};
  \node (ti2) at (4.6,-0.85) {$t_{i+1}$};
  
  \node (c) at (6.2,0.90)  {$c$};
  \node (d) at (6.2,-0.85) {$d$};

  \draw
    (a) -- (ri)
    (ri)--(ti2)
    (ti)--(ti2)
    (b)--(ti)
  ;
  \draw
    (a) -- (si)
    (si)--(si2)
    (si2)--(c)
    (ri)--(ri2)
    (ri2)--(d)
    (ti2)--(d)
  ;

\end{tikzpicture}
}

\newcommand{\structures}{
\begin{tikzpicture}
\node (B) at (8,0) {$\B$};
\node (C) at (5.5,1.5) {$(\C,M)$};
\node (C2) at (5.5,-1.5) {$(\C,M')$};
\node (A) at (3,0.5) {$(\A,L)$};
\node (A2) at (3,-0.5) {$(\A,L')$};
\path [->]
    (C) edge (B)
    (C) edge[bend left] node[above] {$k_n$} (B)
    (C2) edge node[below] {$k_n$} (B)
    (A) edge (B)
    (A2) edge (B)
    (A) edge node[above] {$h$} (C)
    (A2) edge node[above] {$h$} (C2)
;
\draw (6.8,-0.5)--(7.25,-0.75);
\draw (6.8,0.5)--(7.25,0.75);
\draw (5.5,-0.1)--(5.95,-0.35);
\end{tikzpicture}
}

\title
{Conservative CSPs in Logspace are in Symmetric Linear Datalog\\\large(under complexity theoretic assumptions)}
\author{Florian Starke\thanks{
The author has been funded by the European Research Council (Project POCOCOP, ERC Synergy Grant
101071674). Views and opinions expressed are however those of the authors only and do not necessarily reflect
those of the European Union or the European Research Council Executive Agency. Neither the European Union nor
the granting authority can be held responsible for them.
} \\ Institut f\"ur Algebra, TU Dresden}
\date{\today}

\begin{document}

\maketitle

\begin{abstract}
We show that the constraint satisfaction problem of any  finite conservative structure that can not pp-construct  $\stCon$ (whose CSP is NL-complete) or $\lin2$ (whose CSP is Mod$_2$L-complete)  is solved by a symmetric linear Datalog program and therefore in L. 
\end{abstract}


\section{Introduction}

In 2017 the Feder-Vardi conjecture, stating that finite domain CSPs admit a P-NP dichotomy, was solved \cite{ZhukFVConjecture,BulatovFVConjecture}. Similar results for other complexity classes within P have been conjectured but remain unproven.
\begin{conjecture}[\cite{EgriLaroseTessonLogspace}]\label{con:symmetricDatalog}
    Let $\B$ be a finite structure, then the following are equivalent
    \begin{enumerate}
        \item $\csp(\B)$ is solved by a symmetric linear Datalog program,
        \item the polymorphisms of $\B$ contain a Hagemann-Mitschke chain and a 3-4WNU,
        \item $\B$ can neither pp-construct $\stCon$ nor $\lin p$ for any prime $p$.
    \end{enumerate}
    If one of the items is true, then $\csp(\B)$ is in L, otherwise $\csp(\B)$ is NL-hard or Mod$_p$L-hard for some prime $p$.
\end{conjecture}

\begin{conjecture}[\cite{LinearDatalog}]\label{con:linearDatalog}
    Let $\B$ be a finite structure, then the following are equivalent
    \begin{enumerate}
        \item $\csp(\B)$ is solved by a linear Datalog program,
        \item the polymorphisms of $\B$ contain a Kearnes-Kiss chain,
        \item $\B$ can neither pp-construct $\HornSAT$ nor $\lin p$ for any prime $p$.
    \end{enumerate}
    If one of the items is true, then $\csp(\B)$ is in NL, otherwise $\csp(\B)$ is P-hard or Mod$_p$L-hard for some prime $p$.
\end{conjecture}

Kazda proved that Conjecture~\ref{con:symmetricDatalog} follows from Conjecture~\ref{con:linearDatalog} \cite{Kazda-n-permute}.
In 2015 Dalmau, Egri, Hell, Larose, and Rafiey verified Conjecture~\ref{con:symmetricDatalog} for finite conservative structure, whose relations are of arity at most 2 (Theorem 1 in \cite{DalmauLICS15}). In this article we will generalize their result to all finite conservative structures. Section~\ref{sec:higherArity} contains the statements that allows us to make the arguments for binary signatures work for higher arity signatures. The rest of the proof is essentially the same as in~\ref{con:symmetricDatalog}.

\section{Preliminaries}
For $n\geq m\geq 0$ let $[m,n]$ denote the set $\{m,\dots, n\}$ and let $[n]$ denote the set $\{1,\dots, n\}$.
For a tuple $t\in A^k$ and $J\subseteq [k]$ define $t_J\coloneqq (t_j)_{j\in J}$.
For a relation $R\subseteq A^k$ define $\pi_J(R)\coloneqq\{t_J\mid t\in R\}$, $\first(R)\coloneqq\pi_1(R)$, $\last(R)\coloneqq\pi_k(R)$, and $\extremes(R)\coloneqq\pi_{1,k}(R)$. For a second set $I\subseteq [k]$ define $\pi_{I,J}(R)\coloneqq\{(t_I,t_J)\mid t \in R\}$.

\subsection{Structures}
A $\tau$-structure $\B$ is \emph{conservative} if for all $U\subseteq B$ there is a unary relationsymbol $R_U\in\tau$ with $R_U^{\B}=U$. A \emph{primitive positive (pp-)formula} is a first order formula that only uses existential quantification and conjunction of atomic formulas ($\bot$ and $x=y$ are allowed as atomic formulas). A relation $R\subset B^k$ is \emph{pp-definable} in $\B$ if there is a pp-formula $\phi(x_1,\dots,x_k)$ such that $R$ consists of the satisfying assignments of $\phi$ in ${\B}$. 
For a structure $\A$ define $\expansion(\A)$ to be the expansion of $\A$ by all relations $\pi_J(R^{\A})$ for all relationsymbols $R\in\tau$ and all $J\subset [k]$, where $k$ is the arity of $R$. Note that all relations of $\expansion(\A)$ are pp-definable in $\A$.
\begin{align*}
    \text{Throughout this article fix a signature $\tau$ and a finite conservative $\tau$-structure $\B$.}
\end{align*} Let $\maxArity$ be the maximal arity of a relationsymbol in $\tau$. Since adding pp-definable relations does not change the complexity of the CSP, we can without loss of generality assume $\B=\expansion(\B)$.

A list is a map $L\colon A\to2^B$. A $f$ homomorphism from a structure $\A$ to $\B$ \emph{respects} $L$ if $f(x)\in L(x)$ for all $x\in A$.
The \emph{(list) constraint satisfaction problem} of $\B$, denoted $\csp(\B)$, is the set of all $(\A,L)$, where $\A$ is a finite $\tau$-structure and $L\colon A\to2^B$ such that there is a homomorphism from $\A$ to $\B$ respecting $L$. The tuple $(\A,L)$ is called an \emph{instance} of $\csp(\B)$. If there is a homomorphism $f$ from $\A$ to $\B$ respecting $L$, then this instance is called \emph{solvable} and $f$ is called a \emph{solution} of $(\A,L)$. 
For $a,b$ define the relations
\begin{align*}
    &O_{a,b}\coloneqq \{(a,a),(a,b),(b,b)\} &&\text{and}
&&P_{a,b}\coloneqq\{(b,b,b),(a,a,b),(a,b,a),(b,a,a)\}.
\end{align*}
Define the structures $\stCon\coloneqq(\{0,1\},\{0\},\{1\},O_{0,1})$ and $\lin2\coloneqq(\{0,1\},\{0\},\{1\},P_{0,1})$. Note that the relation $\{(0,0,0),(1,1,0),(1,0,1),(0,1,1)\}$ is pp-definable in $\lin2$.
We will need the following well known facts about these two structures.
\begin{theorem}\label{thm:stConIffHM}
    A structure $\A$ can not pp-construct $\stCon$ if and only if its polymorphisms contain a Hagemann-Mitschke chain.
\end{theorem}

\begin{theorem}\label{thm:LinpIff34WNU}
    A  structure $\A$ can not pp-construct $\lin p$ for any prime $p$ if and only if its polymorphisms contain a 3-4WNU.
\end{theorem}

\begin{theorem}\label{thm:Lin2Iff34WNU}
    A conservative structure $\A$ can not pp-construct $\lin2$ if and only if its polymorphisms contain a 3-4WNU.
\end{theorem}

For a definition of  Hagemann-Mitschke chain, 3-4WNU, and pp- constructions see for example \cite{StarkeDiss}.

\subsection{Datalog}
\label{sect:datalog}
For a detailed introduction, see, e.g.,~\cite{BodDalJournal}. 
Let $\rho$ be finite relational signatures such that $\tau \subseteq \rho$. 
A \emph{Datalog program} $\Pi$ is a finite set of rules of the form
$ \phi_0 \; {:}{-} \; \phi_1\wedge\dots\wedge\phi_n$,
where each $\phi_i$ is an atomic $\rho$-formula. The formula $\phi_0$ is called the \emph{head} of the rule, and $\phi_1\wedge\dots\wedge\phi_n$ is called the \emph{body} of the rule. The symbols in $\tau$ are called \emph{EDBs} (\emph{extensional database predicates}) and the other symbols from $\rho\setminus\tau$ are called \emph{IDBs} (\emph{intensional database predicates}). In the rule heads, only IDBs are allowed. There is one special IDB $G$ of arity 0, which is called the \emph{goal predicate}. 
IDBs might also appear in the rule bodies. A \emph{derivation} of $\Pi$ on a finite $\tau$-structure $\A$ is a sequence $\A_0\vdash_{R_0} \dots\vdash_{R_{n-1}} \A_n$ of $\rho$-structures $\A_0,\dots,\A_n$ and rules $R_0,\dots,R_{n-1}$ of $\Pi$ such that 
\begin{itemize}
	\item the $\tau$-reduct of $\A_0$ is $\A$ and $P^{\A_0}=\emptyset$ for all $P\in\rho\setminus\tau$ and
	\item for all $i\in[0,n-1]$ the structure $\A_{i+1}$ is obtained from $\A_i$ by adding a tuple $(s(x_1),\dots,s(x_k))$ to $P^{\A_i}$, where $R_i$ is $P(x_1,\dots,x_k)\; {:}{-} \;\phi$ and $s$ is a satisfying assignment of the variables of $\phi$ to elements of $\A_i$.
\end{itemize}
We say that $\Pi$ can \emph{derive the IDB $P$ on the tuple $t$} if $\A_n$ is obtained from $\A_{n-1}$ by adding $t$ to $P^{\A_{n-1}}$.  
We say that $\Pi$ can \emph{derive the goal predicate (on $\A$)} if it can derive $G$ on the empty tuple $()$, i.e., if $G^{\A_n}\not=\emptyset$. 
We say that $\csp(\B)$ is \emph{solved} by $\Pi$ if the following holds: the goal predicate can be derived by $\Pi$ on a finite $\tau$-structure $\A$ if and only if there is \emph{no} homomorphism from $\A$ to $\B$. 
A Datalog program is called 
\begin{itemize}
    \item \emph{linear} if in each rule, at most one IDB appears in the body (we then assume without loss of generality that in every rule whose body contains an IDB, the IDB is listed first).  
\item \emph{symmetric} if it is linear and for every rule $\phi_0 \; {:}{-} \; \phi_1\wedge\phi_2\wedge\dots\wedge\phi_n$ 
the Datalog program also contains the \emph{reversed rule}  $\phi_1 \; {:}{-} \; \phi_0\wedge\phi_2\wedge\dots\wedge\phi_n$.
\end{itemize}

We state some well known facts about Datalog.
\begin{theorem}[Corollary 4.5 in \cite{StarkeDiss}]\label{symDLIsClosedUnderPPCon}
    If $\B$ pp-constructs $\B'$ and $\csp(\B)$ is solved by symmetric linear Datalog, then $\csp(\B')$ is solved by symmetric linear Datalog. 
\end{theorem}

\begin{theorem}[Theorem~12 in \cite{EgriLT08}]\label{thm:stconNotInSymDL}
    There is no symmetric linear Datalog program that solves $\csp(\stCon)$.
\end{theorem}

\begin{theorem}[\cite{FederVardi}]\label{thm:3lin2NotInDL}
    There is no Datalog program that solves $\csp(\lin2)$.
\end{theorem}

Let $k\geq 0$ and let $\B^\ast$ be the expansion of $\B$ by all relations on $B$. The \emph{canonical symmetric linear Datalog program ($k$-program)} of $\B$ is the Datalog program whose IDBs are all at most $k$-ary relations on $B$ and that contains all linear rules $ \phi_0 \; {:}{-} \; \phi_1\wedge\phi_2\wedge\dots\wedge\phi_n$ such that 
\begin{itemize}
    \item $\phi_1\wedge\phi_2\wedge\dots\wedge\phi_n$ 
    has at most $k$ variables,
    \item the formula $(\phi_1\wedge\phi_2\wedge\dots\wedge\phi_n)\Rightarrow \phi_0$ is valid in $\B^\ast$, and
    \item if $\phi_1$ is an IDB, then $\phi_0\wedge\phi_2\wedge\dots\wedge\phi_n$ has at most $k$ variables and the formula $(\phi_0\wedge\phi_2\wedge\dots\wedge\phi_n)\Rightarrow \phi_1$ is valid in $\B^\ast$.
\end{itemize}
The 0-ary relation $\emptyset$ is the goal predicate of the $k$-program. Note that if $R$ is a relation of $\B$, then $R$ is an EDB and an IDB, we will call $R(x)$ either an (EDB) conjunct or an (IDB) conjunct if we need to distinguish the two cases.
We say that an instance $(\A,L)$ \emph{passes the $k$-test} if the $k$-program can not derive the goal predicate on the structure obtained from  $\A$ by adding $x$ to the unary relation $(L(x))^{\A}$ for $x\in A$. The list $L^{\A}_k\colon A\to 2^B$ maps $x\in A$ to the intersection of all unary IDBs that the $k$-program can derive on $x$. In particular, $L^{\A}_k(x)\subseteq L(x)$ for all $x\in A$. If $\A$ is clear from the context we write $L_k$ instead of $L_k^{\A}$.

\subsection{Walks}
Let $\A$ be a structure. A \emph{walk} is a sequence $\omega=y_1,x_2,y_2,\dots,y_{m-1},x_m,y_m$ of tuples of elements of $A$ such that $x_i$ contains the elements of $y_{i-1}$ and $y_i$ for all $i\in[2,m]$. The width of $\omega$ is $\max\{\m x_i\m\mid i\in[2,m]\}$. 
The walk $\omega$ is \emph{simple} if $\m y_1\m=\dots=\m y_m\m=1$ and $\m y_1\m=\dots=\m y_m\m\leq\maxArity$. 
From $\omega$ we define $\struc^{\A}(\omega)$ as the structure obtained from the disjoint union of $\A\m_{x_2},\dots,$ and $\A\m_{x_m}$ by identifying $y_i$ in $\A\m_{x_i}$ with $y_i$ in $\A\m_{x_{i+1}}$ for all $i\in[2,m-1]$. Let $h_i$ be the map that sends the elements of $y_i$ to their corresponding elements in $\struc^{\A}(\omega)$ and let $h$ be the map from $\struc^{\A}(\omega)$ to $\A$ that send each element to the unique element it corresponds to. Note that $h$ is not necessarily injective. A \emph{realization} of $\omega$ that respects $L$ is a homomorphism $f$ from $\struc^{\A}(\omega)$ to $\B$ with $f(x)\in L(h(x))$. The relation $\real{\omega}{(\A,L)}$ consists of all tuples $f(h_1(y_1),\dots,h_m(y_m))$, where $f$ is realization of $\omega$ respecting $L$. It is easy to see that $\real{\omega}{(\A,L)}$ and $\extremes(\real{\omega}{(\A,L)})$ are pp-definable in $\B$.

A \emph{zigzag expansion} of $\omega$ is a walk $\omega'=y_{i_1},x_{i_2},y_{i_2},\dots,y_{i_{n-1}},x_{i_n},y_{i_n}$ such that
\begin{enumerate}
    \item $i_1,\dots,i_n\in[m]$
    \item $i_1=1$ and $i_n=i_m$, and
    \item $i_{j+1}\in\{i_j-1,i_j+1\}$ for all $j\in[n-1]$.
\end{enumerate}
Define $\zzreal{\omega}{(\A,L)}$ to be the union of all $\real{\omega'}{(\A,L)}$, where $\omega'$ is a zigzag expansion of $\omega$. Note that $\zzreal{\omega}{(\A,L)}$ is not necessary pp-definable in $\B$.
Walks have the following connection to Datalog.
\begin{lemma}[Lemma 3 in \cite{DalmauLICS15}]
    Let $(\A,L)$ be an instance and let $\omega$ be a walk from $x$ to $y$ in $A$ of width at most $k$. Then the $k$-test derives $\last(\zzreal{\omega}{(\A,L)})$ on the tuple $y$.
\end{lemma}
\begin{lemma}[Lemma 2 in \cite{DalmauLICS15}]\label{lem:symDLWalkCharakterization}
    An instance $(\A,L)$ of $\csp(\B)$ passes the $k$-test if and only if $\zzreal{\omega}{(\A,L)}\neq\emptyset$ for all walks $\omega$ in $\A$ of width at most $k$.
\end{lemma}

\subsection{Good Pairs}
Let $a,b\in B$ distinct and let $L\colon A\to2^B$. Then $(a,b)$ is a \emph{good pair relative to $L$} if there is an $x \in A$ with $a,b\in L(x)$ and for all $y\in $, all  $c,d\in L(y)$, and for all pp-definable binary relations $R\subseteq B^2$ we have $\{(a,c),(b,d),(b,c)\}\in R$ implies $(a,d)\in R$. 
\begin{lemma}[Lemma 5 in \cite{DalmauLICS15}]\label{lem:goodPairExists}
Let $\B$ be a finite conservative structure that can not pp-define $O_{a,b}$ on any elements $a,b\in B$. For each list $L\colon A\to 2^B$ with $\m L(x)\m\geq 2$ for some $x\in A$ there exists $x\in A$ and $a,b\in L(x)$ such that $(a,b)$ is a good pair relative to $L$.
\end{lemma}

\section{Main Result}
We are now ready to state the main result.

\begin{theorem}\label{thm:main}
    Let $\B$ be a finite conservative structure, then the following are equivalent
    \begin{enumerate}
        \item $\csp(\B)$ is solved by a symmetric linear Datalog program,
        \item the polymorphisms of $\B$ contain a Hagemann-Mitschke chain and a 3-4WNU,
        \item $\B$ can not pp-construct $\stCon$ or $\lin 2$,
        \item for no $a,b \in B$ distinct can $\B$ pp-define $O_{a,b}$ or $P_{a,b}$.
    \end{enumerate}
    If one of the items is true, then $\csp(\B)$ is in L, otherwise $\csp(\B)$ is NL-hard or Mod$_2$L-hard.
\end{theorem}

The proof will be an induction. We first need some preparation. 
\begin{lemma}\label{lem:inductionnumbersExist}
There exists an increasing sequence $\maxArity=k_0\leq s_1\leq r_1\leq k_1\leq s_2\leq r_2\leq k_2\leq \dots$  such that for all $n\geq 1$ and all instances $(\A,L)$ of $\csp(\B)$
\begin{enumerate}
\item $(\A,L)$ passes the $s_n$-test implies $(\expansion(\A),L_{\maxArity})$ passes the $k_{n-1}$-test,
\item the $r_n$-test does not derive $B\setminus U$ on x implies $(\A,L')$ passes the $s_n$-test, where $L'$ is obtained from $L$ by changing the value at $x$ to $L(x)\cap U$, and
\item $(\A,L)$ passes the $k_n$-test implies  $(\A,L_{r_n})$ passes the $k_{n-1}$-test.
\end{enumerate}
\end{lemma}
The Lemmata needed for the proof will be provided in Section~\ref{sec:DatalogSimulations}
\begin{proof}
We define the sequence recursively. Let $n\geq 1$. Choose $s_n\coloneqq k_{n-1}+2\maxArity$, $r_n\coloneqq s_n+1$, and $k_n\coloneqq r_n+k_{n-1}$.
To verify (1) assume that $(\A,L)$ passes the $s_n$-test. 
Since $s_n=k_{n-1}+2\maxArity-2$, Lemma~\ref{lem:DatalogSimulateSubProgram} implies $(\A,L_{\maxArity})$ passes the $(k_{n-1}+\maxArity)$-test. By  Lemma~\ref{lem:DatalogSimulateExpansion}, $(\expansion(A),L_{\maxArity})$ passes the $k_{n-1}$-test.
Item (2) is the contraposition of Lemma~\ref{lem:DatalogSimulateUnary} and (3) is Lemma~\ref{lem:DatalogSimulateSubProgram}.
\end{proof}

\begin{definition}
For a list $L\colon A\to 2^B$ define the \emph{weight} of $L$ to be 
\[\wt(L)\coloneqq \m\{U\mid x\in A, U\subseteq L(x), \m U\m\geq 2\}\m.\]
For $n\geq 0$ define 
\begin{align*}
    \begin{gathered}
    \text{ every instance $(\A,L)$ of $\csp(\B)$ with $\wt(L)\leq n$}\\
    \text{that passes the $k_n$-test has a solution.}
    \end{gathered}\tag{$\Phi_n$}
\end{align*}
\end{definition}


\begin{lemma}\label{lem:Phi0Holds}
    The statement $\Phi_0$ holds.
\end{lemma}
\begin{proof}
    Let $(\A,L)$ be an instance of $\csp(\B)$ with $\wt(L)=0$ that passes the $k_0$-test. Since it passes the $k_0$-test, no $L(x)$ can be empty. Hence, $\wt(L)=0$ implies that $L(x)$ is a singleton for every $x\in A$. Therefore, $L$ can be seen as a map from $\A$ to $\B$. Since $k_0=\maxArity$ and $(\A,L)$ passes the $k_0$-test the map $L$ provides a homomorphism from $\A$ to $\B$.
\end{proof}

\begin{lemma}\label{lem:replaceB}
Let $n\geq1$ and let $\B$ be a finite conservative structure for which $\Phi_{n-1}$ holds. 
Let $(\A,L)$ be an instance of $\csp(\B)$ with $\wt(L)\leq n$,  $(a,b)$ be a good pair relative to $L$, and $\emptyset\neq X\subseteq A$ such that
\begin{enumerate}
    \item $\B$ can neither pp-define $O_{a,b}$ nor $P_{a,b}$,
    \item $a,b\in L_{r_n}(x)$ for every $x\in X$, and
    \item every simple walk $\omega$ between elements of $X$ satisfies
    \begin{align*}
        (b,b)\in\extremes(\zzreal{\omega}{(\A,L)}).
    \end{align*}
\end{enumerate}
Then for every solution $f$ and every $x\in X$ with $f(x)=b$ there is a solution $g$ with $g(x)=a$ and $g(y)\neq b$ for all $y\in X$ with $f(y)\neq b$. 
\end{lemma}
The proof of this lemma will be given in Section~\ref{sec:replaceB}. 

\begin{lemma}\label{lem:minimalCounterExample}
Let $n\geq1$ and let $\B$ be a finite conservative structure for which $\Phi_{n-1}$ holds. If $\Phi_n$ does not hold, there is a satisfiable instance $(\A,L)$ of $\csp(\B)$ with $\wt(L)\leq n$, a good pair $(a,b)$ relative to $L$, and a nonempty $X\subseteq A$ such that 
\begin{enumerate}
    \item $a,b\in L_{r_n}(x)$ for every $x\in X$ and
    \item every simple walk $\omega$ between elements of $X$ satisfies
    \begin{align*}
        (b,b)\in\extremes(\zzreal{\omega}{(\A,L)}).
    \end{align*}
    \item every solution $f$ satisfies $b\in f(X)$.
\end{enumerate}
\end{lemma}
\begin{proof}
    Assume $\Phi_n$ does not hold, then there exists  an instance $(\C,M)$ of $\csp(\B)$ with $\wt(M)\leq n$ such that $(\C,M)$ has no solution and $(\C,M)$ passes the $k_n$-test. Assume without loss of generality that $M$ is minimal in the sense that removing any element from any list will lead to a failure of the $k_n$-test. By Lemma~\ref{lem:Phi0Holds} $\wt(M)\geq 1$. Hence, Lemma~\ref{lem:goodPairExists} implies that there are $a,b\in B$ such that $(a,b)$ is a good pair relative to $M$. 
    Assume that $\{a,b\}\not\subseteq M_{r_n}(y)$ for all $y\in C$. Then $\wt(M_{r_n})\lneq\wt(M)$. By Lemma~\ref{lem:inductionnumbersExist}(3), $(\C,M_{r_n})$ passes the $k_{n-1}$-test. Hence, $\Phi_{n-1}$ implies that $(\C,M_{r_n})$ has a solution $f$. This $f$ would also be a solution of $(\C,M)$. A contradiction. Therefore, there exists an $x\in C$ with $a,b\in M_{r_n}(x)$.

    \begin{figure}
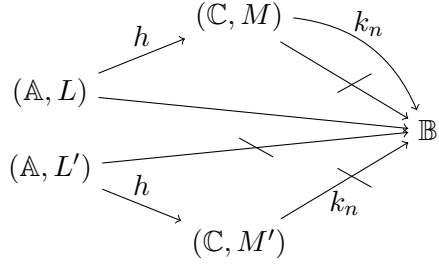

        \centering
        \structures
        \caption{The structures involved in the proof of Lemma~\ref{lem:minimalCounterExample}. Edges labelled with $k_n$ indicate that the $k_n$-test does not refute the existence of a  homomorphism.}
        \label{fig:structures}
    \end{figure}
    Define $M'$ to be the list obtained from $M$ by changing the value at $x$ to $M(x)\setminus\{b\}$. By minimality of $M$ the instance $(\C,M')$ does not pass the $k_n$-test. By Lemma~\ref{lem:symDLWalkCharakterization} there is a walk $\omega$ in $\C$ of width at most $k_n$ such that $\zzreal{\omega}{(\C,M')}=\emptyset$. Since $(\C,M)$ passes the $k_n$-test $\zzreal{\omega}{(\C,M)}\neq\emptyset$. Hence, there is a zigzag expansion $\omega'$ of $\omega$ with $\real{\omega'}{(\C,M)}\neq\emptyset$. Let $h$ be the homomorphism from 
    $\A\coloneqq (\struc^{\C}(\omega'))$ to $\C$ provided by the construction of $\A$. Define $L\colon A\to 2^B, y\mapsto M(h(y))$, $L'\colon A\to 2^B, y\mapsto M'(h(y))$, and $X\coloneqq h^{-1}(x)$.
    Since $\real{\omega'}{(\C,M)}\neq\emptyset$ and $\zzreal{\omega}{(\C,M')}=\emptyset$ we have that $(\A,L)$ has a solution and $(\A,L')$ has no solution, see Figure~\ref{fig:structures} for a depiction of the structures and their relationships. 
    Therefore, for every solution $f$ of $(\A,L)$ there must be a $y\in X$ with $f(y)=b$, otherwise $f$ would be a solution of $(\A,L')$. Hence, (3) holds. In particular, $X\neq\emptyset$.
    
    To show (1) note that every walk $\eta$ in $\A$ gives, via $h$, a walk $\eta'$ in $\C$ such that $\real{\eta'}{(\C,M)}\subseteq \real{\eta}{(\A,L)}$. Hence, $M_{r_n}(h(y))\subseteq L(y)$ for all $y\in A$. In particular, $a,b\in M_{r_n}(x)=M_{r_n}(h(y))\subseteq L_{r_n}(y)$ for all $y\in X$.
    For (2) let $\eta$ be a simple walk between $x_1\in X$ and $x_2\in X$. Let $\eta'$ be the corresponding walk obtained via $h$ from $h(x_1)=x$ to $h(x_2)=x$. Since $\eta'$ is simple, $r_n\geq\maxArity$, and $b\in M_{r_n}(x)$ we get $(b,b)\in\extremes(\zzreal{\eta'}{(\C,M)})\subseteq \extremes(\zzreal{\eta}{(\A,L)})$.
\end{proof}

Combining the previous two lemmata we can prove the main theorem.
\begin{proof}[Proof of Theorem~\ref{thm:main}]
    The equivalence of $(2)$ and $(3)$ follows from Theorems~\ref{thm:stConIffHM} and~\ref{thm:Lin2Iff34WNU}. The direction $(1)\Rightarrow (3)$ follows from Theorems~\ref{thm:3lin2NotInDL} and~\ref{thm:stconNotInSymDL}. Clearly, $(3)$ implies $(4)$. 
    
    $(4)\Rightarrow(1)$: Assume there is an $n\geq$ such that $\Phi_n$ does not hold. By Lemma~\ref{lem:Phi0Holds}, $\Phi_0$ holds. Hence, $n\geq 1$ and Lemma~\ref{lem:minimalCounterExample} provides a satisfiable instance $(\A,L)$ and a set $X\subseteq A$. By (4) these now satisfy all the assumptions of Lemma~\ref{lem:replaceB}. Let $f$ be a solution of $(\A,L)$. Repeatedly applying Lemma~\ref{lem:replaceB} yields a solution $g$ with $g(y)\neq b$ for all $y\in X$. This contradicts Lemma\ref{lem:minimalCounterExample} stating that $b\in g(X)$.
    Therefore, $\Phi_n$ holds for all $n\geq 0$. Since the weight of a list is bounded by $N\coloneqq \m\{U\subseteq B\mid \m U\m\geq 2\}\m=2^{\m B\m}-\m B\m-1$ we conclude that the $k_N$-test solves $\csp(\B)$.
\end{proof}

In the following two sections we will prove Lemma~\ref{lem:replaceB}.

\section{Binary failure in higher arity replacement}\label{sec:higherArity}
In this section we produce the results that will enable us to lift the approach Dalmau, Egri, Hell, Larose, and Rafiey used for conservative structures with binary signature to conservative structures higher-ary signatures. This is also the only part where the assumption that $\B$ can not pp-define $P_{a,b}$ will be used explicitly.

\begin{lemma}\label{lem:replacingBinTuples}
Let $\B$ be a finite conservative structure, $k\geq1$, $R$ be a pp-definable $(k+1)$-ary relation, and $(a,b)$ be a good pair relative to $L\colon i\mapsto \pi_i(R)$ such that $\B$ can neither pp-define $O_{a,b}$ nor $P_{a,b}$. Then for all $(b,t)\in R$ with $t_i\in\pi_i(R_a)$ for all $i\in[k]$ we have $(a,t)\in R$.
\end{lemma}
\begin{proof}
    Suppose there exists $t=(t_1,\dots,t_k)\in R_b\setminus R_a$ with $t_i\in\pi_i(R_a)$ for all $i$.  
    Let $J\subseteq [k]$ be minimal (with respect to inclusion) with $t_J\notin P_a\coloneqq\pi_J({R_a})$. Since $t_i\in\pi_i(R_a)$ we have $|J|\geq 2$. Minimality of $J$ implies that for each $i\in J$ there is an $s_i\in B$ such that the tuple $t_J[s_i]\colon J\to B, i\mapsto s_i,j\mapsto t_j$ for all $j\in J\setminus\{i\}$ belongs to $P_a$. For all $I\subseteq J$ define 
    \[T_I\coloneqq\{(c,r_I)\mid (c,r)\in R, r_j=t_j\text{ for all }j\in J\setminus I\}\cap\left(\{a,b\}\times\prod_{j\in I}\{t_j,s_j\}\right).\]
    Note that, since $\B$ is conservative, $T_I$ is pp-definable. 
    Consider $i\in J$. Then $t\in R_b$ and $t_J[s_i]\in P_a$ imply $(b,t_i),(a,s_i)\in T_{\{i\}}$. Since $t_J\notin P_a$ we have $(a,t_i)\notin T_{\{i\}}$.
    Note that $(a,b)$ is a good pair relative to the (possibly) smaller list $(\{a,b\},\{t_i,s_i\})$. Hence, since $T_{\{i\}}$ is pp-definable, $(b,s_i)\notin T_{\{i\}}$ and $T_{\{i\}}=\{(b,t_i),(a,s_i)\}$ is the graph of a bijection $\theta_i\colon\{a,b\}\to\{t_i,s_i\}$.

    Let $i\lneq j\in J$. Define $S\coloneqq\{(c,c_i,c_j)\mid (c,\theta_i(c_i),\theta_j(c_j))\in T_{\{i,j\}}\}$ and observe that $S$ is pp-definable. Note that we have the following implications
    \begin{align*}
        &\text{$(b,t)\in R$ implies $(b,b,b)\in S$,}
         &&\text{$t_J\notin P_a$ implies $(a,b,b)\notin S$,}\\
         &\text{$t_J[s_i]\in P_a$ implies $(a,b,a)\in S$,}
         &&\text{$(b,s_i)\notin T_{\{i\}}$ implies $(b,a,b)\notin S$,}\\ 
         &\text{$t_J[s_j]\in P_a$ implies $(a,a,b)\in S$,}
         &&\text{$(b,s_j)\notin T_{\{j\}}$ implies $(b,b,a)\notin S$.}
    \end{align*}
    For the two remaining tuples we have
    \begin{itemize}
        
        
        \item $(b,a,a)\in S$, otherwise $\{(c,c_i)\mid (c,c_i,c_j)\in S\}=\{(b,b),(a,b),(a,a)\}=O_{a,b}$ is pp-definable
        \item $(a,a,a)\notin S$, otherwise $\{(c,c_j)\mid (c,c,c_j)\in S\}=\{(b,b),(a,b),(a,a)\}=O_{a,b}$ is pp-definable
    \end{itemize}
    Therefore, $S=\{(b,a,a),(a,b,a),(a,a,b),(b,b,b)\}=P_{a,b}$ is pp-definable. A contradiction.
\end{proof}

\begin{lemma}\label{lem:getBinaryWitness}
Let $\B$ be a finite conservative structure, $k\geq1$, $R$ be a pp-definable $k$-ary relation, $s,t\in R$, $\emptyset\neq I\subseteq[k]$, $J=[k]\setminus I$, and $(a,b)$ be a good pair relative to $L\colon i\mapsto \pi_i(R)$ such that $\B$ can neither pp-define $O_{a,b}$ nor $P_{a,b}$ and for all $i\in I$ the set $\{(a,s_i),(b,t_i)\}$ is pp-definable. 
If $(s_I,t_J)\notin \pi_{I,J}(R)$ then there are $i\in I$ and $j\in J$ such that $(s_i,t_j)\notin\pi_{i,j}(R)$. In particular, there is no $r_j$ with $(s_i,r_j)\in\pi_{i,j}(R)$ and $(t_i,r_j) \in\pi_{i,j}(R)$. 
\end{lemma}
\begin{proof}
    For $i\in[k]$ define $\theta_i\colon\{a,b\}\to\{s_i,t_i\},a\mapsto s_i,b\mapsto t_i$. By assumption the graph of $\theta_i$ is pp-definable for all $i\in I$. Therefore, the relations
    \begin{align*}
        Q&\coloneqq\{(c,r_J)\mid c\in\{a,b\}, (\theta(c),r_J)\in \pi_{I,J}(R)\}\text{ and}\\
        Q_{j}&\coloneqq \{(c,r_j)\mid (c,r_J)\in Q\} \text{ for }j\in J
    \end{align*}
    where $\theta(a)=s_I,\theta(b)=t_I$, are pp-definable. 
    By definition $t\in R$ and $(\theta(a),t_J)=(s_I,t_J)\notin \pi_{I,J}(R)$. Hence $(b,t_J)\in Q$ and $(a,t_J)\notin Q$.
    By Lemma~\ref{lem:replacingBinTuples}, there must exist $j\in J$ such that $(a,t_j)\notin Q_{j}$. Since $s,t\in R$ we have $(a,s_j),(b,t_j)\in Q_j$. Because $(a,b)$ is good relative to $L$ and $\{s_j,t_j\}\subseteq L(j)$ we conclude $(b,s_j)\notin Q_j$. Hence, $Q_j$ is the graph of $\theta_j$ and the relations
     \begin{align*}
     T&\coloneqq \{(c,r_I)\mid c\in\{a,b\},r\in R, \theta_j(c)=r_j\}\\    
     T_i&\coloneqq\{(c,r_i)\mid (c,r_I)\in T\}\text{ for }i\in I
     \end{align*}
     are  pp-definable. By definition $s,t\in R$ imply $(b,t_I)\in T$. 
     Assume $(a,t_I)\in T$. Then there is an $r\in R$ such that $r_j=s_j$ and $r_I=t_I$. Hence, $(\theta(b),r_J)=(t_I,r_J)\in\pi_{I,J}(R)$. Therefore, $(b,r_J)\in Q$ and $(b,s_j)=(b,r_j)\in Q_j$, a contradiction. 
     We conclude that $(a,t_I)\notin T$. By Lemma~\ref{lem:replacingBinTuples} there is an $i\in I$ with $(a,t_i)\notin T_i$. Since $s,t\in R$ we have $(a,s_i),(b,t_i)\in T_i$. Because $(a,b)$ is good relative to $L$ and $\{s_i,t_i\}\subseteq L(i)$ we conclude $(b,s_i)\notin T_j$. 
     Therefore, $(s_i,t_j)\notin\pi_{i,j}(R)$. Since otherwise there would be a tuple $r\in R$ with $r_i=s_i$ and $r_j=t_j=\theta_j(b)$ contradicting $(b,s_i)\notin T_j$.
     
     For the final part, assume there is an $r_j$  with $(s_i,r_j)\in\pi_{i,j}(R)$ and $(t_i,r_j) \in\pi_{i,j}(R)$. Since $(t_i,t_j)\in\pi_{i,j}(R)$ the pp-definable relation  $\{(c,q_j)\mid (\theta_i(c),q_j)\in\pi_{i,j}(R)\}$ contains $(a,r_j),(b,r_j)$, and $(b,t_j)$. Because $(a,b)$ is a good pair relative to $L$ and $\{t_j,r_j\}\subseteq L(j)$ the relation also contains $(a,t_j)$. Since $s_i=\theta_i(a)$, this contradicts $(s_i,t_j)\notin\pi_{i,j}(R)$.
\end{proof}

\section{Replacing $b$ by $a$ in a solution}\label{sec:replaceB}
In this section we will prove Lemma~\ref{lem:replaceB}. 
Throughout this section fix 
an instance $(\A,L)$ of $\csp(\B)$, a good pair $(a,b)$ relative to $L$, a nonempty set $X\subseteq A$, a solution $f$, and an $x\in X$ with $f(x)=b$ such that all the assumptions of Lemma~\ref{lem:replaceB} are satisfied. 
First we define the region that is affected by changing the value of $f$ at $x$ to $a$.
\begin{definition}
A simple walk $\omega$ in $\A$ is \emph{separating} if there is no $c\in B$ with $(a,c),(b,c)\in\extremes(\zzreal{\omega}{(\A,L)})$. 
Define the set
\[Z\coloneqq\{z\in A\mid \text{there is a separating walk from $x$ to $z$}\}.\]
\end{definition}

\begin{observation}\label{obs:separatingWalks}
We make a few simple observations about $Z$ and separating walks.
    \begin{enumerate}
        \item Note that the walk $x$ is a separating from $x$ to $x$. Hence $x\in Z$.  
        \item For each separating walk $\omega$ from $x$ to $y\in A$ there are $c,d\in L(y)$ with $(a,c),(b,d)\in \extremes(\zzreal{\omega}{(\A,L)})$, otherwise, since $r_n\geq\maxArity$, $a$ and $b$ would not be in $L_{r_n}(x)$.
        \item If we change the value of $f$ at $x$ to $a$ then the values of $f$ at elements of $Z$ must also change. 
    \end{enumerate}
\end{observation}
\begin{definition}
    For a list $M\colon A\to 2^B$ define $\trace_{Z}(\A,M)$ as $(\expansion(\A)|_Z,M|_Z)$, where $\expansion(\A)|_Z$ is the by $Z$ induced substructure of $\expansion(\A)$. 
\end{definition}
The expansion makes sure that $\trace_Z(\A,M)$ can still see a trace of how $Z$ connects to the rest of $A$. 
We want to replace the values of $f$ at elements of $Z$ with those of a solution of $\trace_Z(\A,M)$. To make sure that the new map sends $x$ to $a$ we can not take $M=L$ but instead need a new list.
Fix the list $L^a\colon A\to 2^B$ with $L^a(x)=a$ and $L^a(y)=L(y)$ for all $y\in A\setminus\{x\}$. Recall that $\maxArity$ denotes the maximal arity of a relation of $\B$.

\begin{lemma}
    The instance $\trace_Z(\A,L^a_{\maxArity})$ has a solution.
\end{lemma}
\begin{proof}
Since $(\A,L)$ passes the $k_n$-test, by Lemma~\ref{lem:inductionnumbersExist}(3) $(\A,L_{r_n})$ passes the $k_{n-1}$-test. By assumption of Lemma~\ref{lem:replaceB} $a\in L_{r_n}(x)$. Hence, the $r_n$-test can not derive $B\setminus\{a\}$ on $x$. By Lemma~\ref{lem:inductionnumbersExist}(2), 
$(\A,L^a)$ passes the $s_n$-test.
By item (1) of Lemma~\ref{lem:inductionnumbersExist}
$(\expansion(\A),L^a_{\maxArity})$ passes the $k_{n-1}$-test.
Since $\trace_Z(\A,L^a_{\maxArity})$ is a substructure of $(\expansion(\A),L^a_{\maxArity})$ it also passes the $k_{n-1}$-test.

We will now show $L^a_{\maxArity}(z)\subsetneq L_{r_n}(z)$ for all $z\in Z$. First note that $L^a(z)\subseteq L(z)$ for all $z\in Z$. Since $r_n\geq\maxArity$ we conclude $L^a_{\maxArity}(z)\subseteq L_{r_n}(z)$ for all $z\in Z$.
We have $L^a_{\maxArity}(x)=\{a\}\subsetneq \{a,b\}\subseteq L_{r_n}(x)$. Let $z\in Z\setminus \{x\}$ . By definition of $Z$ there is a separating walk $\omega$ from $x$ to $z$. 
Define $c\coloneqq f(z)$. 
Then $(b,c)=(f(x),f(z))\in\extremes(\real{\omega}{(\A,L_{r_n})})$. 
Assume $c\in L^a_{\maxArity}(z)$. 
Since $\omega$ is a simple walk the $\maxArity$-test derives $\last(\zzreal{\omega}{(\A,L^a)})$ on $z$. Hence, $c\in\last(\zzreal{\omega}{(\A,L^a)})$. Because $L^a_{\maxArity}(x)=\{a\}$ we therefore have $(a,c)\in\extremes(\zzreal{\omega}{(\A,L^a)})$. Since $L^a$ and $L_{r_n}$ are both stricter lists than $L$ we have $(a,c),(b,c)\in\extremes(\zzreal{\omega}{(\A,L)})$, contradicting that $\omega$ is separating. Therefore, $c\notin L^a_{\maxArity}(z)$. 
See Figure~\ref{fig:separatingPotato}  for a depiction of the situation.
\begin{figure}
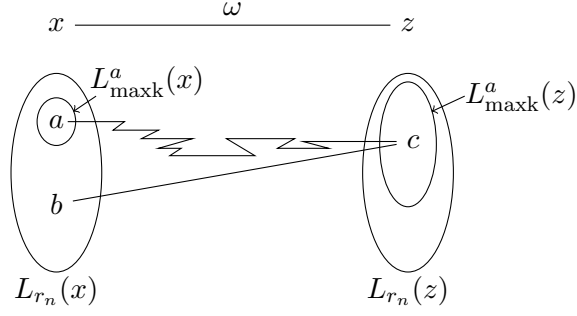

    \centering
    \separatingPotato
    \caption{A depiction of how a separating walk reduces the list of $z$.}
    \label{fig:separatingPotato}
\end{figure}
By definition of weight we obtain \[\wt(L^a_{\maxArity})\lneq\wt(L_{r_n})\leq\wt(L)\leq n.\]    
In conclusion, $\trace_Z(\A,L^a_{\maxArity})$ passes the $k_{n-1}$ and its list has weight $\leq n-1$. Therefore,  $\Phi_{n-1}$ implies that $\trace_Z(\A,L^a_{\maxArity})$ has a solution.
\end{proof}

Fix a solution $h\colon Z\to B$ of $\trace_Z(\A,L^a_{\maxArity})$ and define the map
\begin{align*}
    g\colon A&\to B\\
    y&\mapsto\begin{cases}
        h(y) &\text{if }y\in Z\\
        f(y) &\text{if }y\notin Z
    \end{cases}
\end{align*}
Before we can show that $g$ is the desired solution in Lemma~\ref{lem:replaceB} we need to show some more properties of the region $Z$.

\begin{lemma}\label{lem:separatingWalksStayInZ}
    Let $\omega$ be a simple walk from $x$ to $y\in A$ and $\omega'$ a walk from $y$ to $z\in Z$. If $\omega\omega'$ is separating, then $\omega$ is also separating. In particular, $y\in Z$.
\end{lemma}
\begin{proof}
    We show the contraposition. Assume  $\omega$ is not separating. By definition there is a $c\in B$ with $(a,c),(b,c)\in\extremes(\zzreal{\omega}{(\A,L)})$. By Observation~\ref{obs:separatingWalks} there is a $c'\in L(z)$ with $(a,c')\in\extremes(\zzreal{\omega\omega'}{(\A,L)})$. Therefore, there is a zigzag-realization of $\omega\omega'$ going from $b$ over $c$ and $a$ to $c'$ and $(b,c')\in\extremes(\zzreal{\omega\omega'}{(\A,L)})$. Hence, $\omega\omega'$ is not separating. 
\end{proof}

\begin{lemma}\label{lem:withinZBijectionsArePPDefinable}
    Let $z\in Z$. Then $\{(a,h(z)),(b,f(z))\}$ is pp-definable in $\B$.
\end{lemma}
\begin{proof}
    Let $\omega=y_1,x_2,y_2,\dots,y_{m-1},x_m,y_m$ be a separating walk from $x$ to $z$. The relation $R\coloneqq \extremes(\real{\omega}{(\A,L)})\cap(\{a,b\}\times\{g(z),f(z)\})$ is pp-definable. 
    Note that, by Lemma~\ref{lem:separatingWalksStayInZ}, $y_i\in Z$ for all $i$. Hence, $h(y_i)$ is defined and applying $f$ and $h$ to $\omega$ yields $(a,h(z))\in R$ and $(b,f(z))\in R$, respectively. 
    If $(a,f(z))\in R$ or $(b,h(z))\in R$, then $\omega$ would not be separating. Hence, $R=\{(a,h(z)),(b,f(z))\}$, as desired.
\end{proof}

Let $\omega$ be a simple walk from $x$ to $y\in A$, $c,d\in L(y)$, and let $a=s_1,\dots,s_m=c$  and $b=t_1,\dots,t_m=d$ be two realizations of $\omega$. These two realizations \emph{avoid} each other if $\{(s_i,t_{i+1}),(t_i,s_{i+1})\}\cap\pi_{i,i+1}(\real{\omega}{(\A,L)})=\emptyset$ for all $i\in[m-1]$.
\begin{lemma}\label{lem:avoidingWalks}
    Let $\omega$ be a simple walk from $x$ to $y\in A$ and $c,d\in L(y)$. If there are realizations $a=s_1,\dots,s_m=c$  and $b=t_1,\dots,t_m=d$ of $\omega$ that avoid each other
    , then $(a,d)\notin\extremes(\real{\omega}{(\A,L)})$. 
\end{lemma}
\begin{proof}
Let $\omega=y_1,x_2,y_2,\dots,y_{m-1},x_m,y_m$, $s\coloneqq s_1,\dots,s_m$, and $t\coloneqq t_1,\dots,t_m$. Assume there is a realization $a=r_1,\dots,r_m=d$ of $\omega$.
    Let $i\in[m-1]$ be minimal with $\{(r_i,t_{i+1})),(t_i,r_{i+1})\}\cap\pi_{i,i+1}(\real{\omega}{(\A,L)})\neq\emptyset$. 
    Assume $(r_i,t_{i+1})$ is contained in $\pi_{i,i+1}(\real{\omega}{(\A,L)})$, see Figure~\ref{fig:abcd} for a visualization. Define the walk \[\eta \coloneqq y_1,x_2,y_2,\dots,y_{i},x_{i+1},y_{i+1},x_{i+1},y_i,\dots,y_2,x_2,y_1.\] 
    \begin{figure}
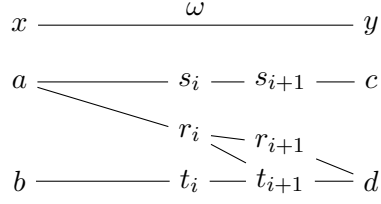

        \centering
        \abcd
        \caption{Realizations of the walk $\omega$.}
        \label{fig:abcd}
    \end{figure}
    Consider the pp-definable relation \begin{align*}
        R\coloneqq \real{\eta}{(\A,L)}\cap(\{s_1,r_1,t_1\}\times\dots\times\{s_i,r_i,t_i\}\times\{s_{i+1},t_{i+1}\}\times\{s_i,t_i\}\times\dots\times\{s_1,t_1\})
    \end{align*} 
    Observe that 
    \begin{align*}
        &(a,s_2,\dots,s_i,s_{i+1},s_i,\dots,s_2,a)\in R,\\
        &(a,r_2,\dots,r_i,t_{i+1},t_i,\dots,t_2,b)\in R,\text{ and}\\
        &(b,t_2,\dots,t_i,t_{i+1},t_i,\dots,t_2,b)\in R.
        \intertext{Assume $(b,a)\in\extremes(R)$, then }
        &(b,t_2,\dots,t_i,t_{i+1},r_i,\dots,s_2,a)\in R.
    \end{align*}
    Therefore, $r_i\in\{s_i,t_i\}$. 
    If $r_i=s_i$, then $(s_i,t_{i+1})=(r_i,r_{i+1})\in\pi_{i,i+1}(\real{\omega}{(\A,L)})$, contradicting that $s$ and $t$ are avoiding.
    If $r_i\neq s_i$, then $i\geq2$ and $r_i=t_i$. Hence, $(r_{i-1},r_i)=(r_{i-1},t_i)$, contradicting the minimality of $i$.
    Therefore the pp-definable relation $\extremes(R)$ is equal to $\{(a,a),(a,b),(b,b)\}=O_{a,b}$, a contradiction.

    In the second case $(f(y_i'),r_{i+1})$ is contained in $\pi_{i,i+1}(\real{\omega'}{(\A,L)})$, then the pp-definable relation \[\pi_{1,i+1}(\real{\omega'}{(\A,L)}\cap(\{r_1,f(y_1)\}\times\dots\times \{r_m,f(y_m')\}))\] contains $(b,r_{i+1}),(b,f(y_{i+1}')),(a,r_{i+1})$. Since $(a,b)$ is a good pair relative to $L$, the relation must also contain $(a,f(y_{i+1}'))$. Because $i$ is minimal we conclude $(r_i,f(y_{i+1}'))\in \pi_{i,i+1}(\real{\omega'}{(\A,L)})$ and we are back in the first case.  
\end{proof}

\begin{lemma}\label{lem:gIsSolution}
    The map $g$ is a solution of $(\A,L)$.
\end{lemma}
\begin{proof}
    First we verify that $g$ respects $L$. For $z\in Z$ we have $g(z)=h(z)\in L^a_{\maxArity}(z)\subseteq L(z)$ and for $y\in A\setminus Z$ we have $g(y)=f(y)\in L(y)$. 
    Now we check that $g$ preserves all relations. Let $y$ be a tuple in some $k$-ary relation $R^{\A}$. The goal is to show that $g(y)\in R^{\B}$. Define $I\coloneqq\{i\in[k]\mid y_i\in Z\}$, the set of indices inside $Z$, and $O\coloneqq[k]\setminus I$, the set of indices outside $Z$. 
    If $I=\emptyset$, then $g(y)=f(y)\in R^{\B}$. Analogously, if $O=\emptyset$, then $g(y)=h(y)\in R^{\expansion(\A)}=R^{\A}$.
    Consider the case $I\neq\emptyset$ and $O\neq\emptyset$. 
    Define $t=f(y)$. Since $g$ preserves the relations of  $\expansion(\A)$ we have $g(y_I)\in \pi_I(R^{\A})$. Hence, there exits a tuple $s\in R^{\A}$ with $s_I=g(y_I)$. 
    If $(g(y_I),g(y_O))=(h(y_I),f(y_O))=(s_I,t_O)\in \pi_{I,O}(R^{\A})$, then $g(y)\in R$ as desired.
    Assume $(s_I,t_O)\notin \pi_{I,O}(R^{\A})$. By Lemma~\ref{lem:withinZBijectionsArePPDefinable}, $\{(a,s_i),(b,t_i)\}$ is pp-definable for all $i\in I$. Hence, all assumptions of Lemma~\ref{lem:getBinaryWitness} are satisfied and we obtain $i\in I$ and $o\in O$ for which
    \begin{align*}
        \text{there is no $r_o\in B$ with $(s_i,r_o)\in\pi_{i,o}(R^{\B})$ and  $(t_i,r_o)\in\pi_{i,o}(R^{\B})$.}\tag{$\ast$}\label{tag:SiRo}
    \end{align*} 
    \begin{figure}
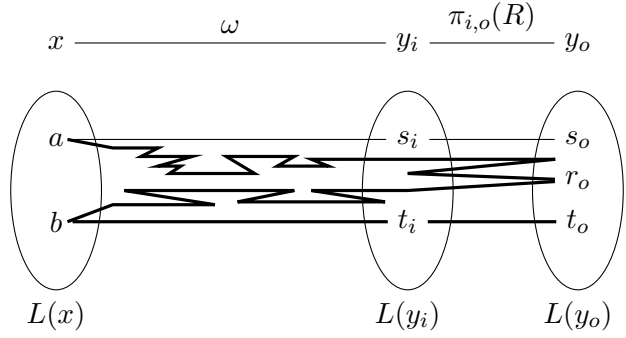

        \centering
        \walkExtension
        \caption{Extension of the walk $\omega$, a zigzag expansion of $\omega,y,y_o$ connecting $a$ and $t_o$ is highlighted in bold.}
        \label{fig:walkExtension}
    \end{figure}
    Consider the simple walk $\omega,y,y_o$.
    Since $(y_i,y_o)\in\pi_{i,o}(R^{\A})$, this is a simple walk from $x$ to $y_o$. Therefore, $(a,s_o),(b,t_o)\in\extremes(\zzreal{\omega,y,y_o}{(\A,L)})$.
    Because $y_o\notin Z$ we have that $\omega,y,y_o$ is not separating. Hence, there is an $r_o\in L(y_o)$ with $(a,r_o),(b,r_o)\in\extremes(\zzreal{\omega,y,y_o}{(\A,L)})$, see Figure~\ref{fig:walkExtension}. Since $(b,t_o)$ also belongs to this set \[\text{$\extremes(\zzreal{\omega,y,y_o}{(\A,L)})$ contains $(a,t_o)$.}\] Therefore, there exists a zigzag-expansion $\omega'=y_1',x_2',y_2',\dots,y_{m-1}',x_m',y_m'$ of $\omega,y,y_o$ with $(a,t_o)\in \extremes(\real{\omega'}{(\A,L)})$. Let $a=r_1,\dots,r_m=t_o$ be a realization of $\omega'$ that witnesses this containment. 
    Note that if we define $h(y_o)\coloneqq s_o$, then $b=f(y_1'),\dots,f(y_m')=t_o$ and $a=h(y_1'),\dots,h(y_m')=s_o$ are also realizations of $\omega'$. If 
    these two realizations avoid each other, then $(a,t_o)\in \extremes(\real{\omega'}{(\A,L)})$  yields a contradiction to Lemma~\ref{lem:avoidingWalks}. Proving that $(s_I,t_O)\in \pi_{I,O}(R^{\A})$ as desired.

    Let $j\in[m-1]$. If $\{y'_j,y_{j+1}'\}=\{y_i,y_o\}$, then \[\{(f(y_i'),h(y_{i+1}')),(h(y_i'),f(y_{i+1}'))\}\subset\{(s_i,t_{i+1}),(t_i,s_{i+1}),(t_{i+1},s_i),(s_{i+1},t_i)\}\] which has empty intersection with $\pi_{i,i+1}(\real{\omega'}{(\A,L)})$ by (\ref{tag:SiRo}).
    Now consider the case $\{y'_j,y_{j+1}'\}\neq\{y_i,y_o\}$. Since $y_o$ is the last element on the walk $\omega$ we have $y_o\notin\{y'_j,y_{j+1}'\}$. Therefore, $y'_j,y_{j+1}'\in Z$. If $\{(f(y_i'),h(y_{i+1}')),(h(y_i'),f(y_{i+1}'))\}\cap \pi_{i,i+1}(\real{\omega'}{(\A,L)})\neq\emptyset$, then $(a,t_i)$ or $(b,s_i)$ are contained in $\extremes(\zzreal{\omega}{(\A,L)})$ contradicting that $\omega$ is separating. Hence, the two realizations are avoiding. 
\end{proof}

\begin{lemma}\label{lem:gAvoidsB}
    The map $g$ satisfies $g(x)=a$ and $g(y)\neq b$ for all $y\in X$ with $f(y)\neq b$.
\end{lemma}
\begin{proof}
    Since $x\in Z$ we have $g(x)=h(x)=a$. Let $y\in X$ with $f(y)\neq b$. If $y \notin Z$, then $g(y)=f(y)\neq b$. Consider the case $y\in Z$. By definition of $Z$ there exists a separating walk $\omega$ from $x$ to $y$. Since $x,y\in X$ the assumptions of Lemma~\ref{lem:replaceB} imply $(b,b)\in\extremes(\zzreal{\omega}{(\A,L)})$. Assume $g(y)=b$. Then $b=g(y)=h(y)\in L^a_{\maxArity}(y)$. Hence, $(a,b)\in\extremes(\zzreal{\omega}{(\A,L^a_{\maxArity})})$. 
    This contradicts that $\omega$ is separating. Hence, $g(y)\neq b$ as desired.
\end{proof}
Combining Lemmata~\ref{lem:gIsSolution} and~\ref{lem:gAvoidsB} proves Lemma~\ref{lem:replaceB}.

\section{Datalog Simulations}\label{sec:DatalogSimulations}
In this section we give some technical results on how $k$-programs can simulate other $k'$-programs.
\begin{lemma}\label{lem:DatalogSimulateSubProgram}
Let $(\A,L)$ be an instance.
If $(\A,L)$ passes the $(k+r)$-test, then $(\A,L_{r})$ passes the $k$-test.
\end{lemma}
\begin{proof}
    We show that whenever the $k$-program can derive the $P$ on the tuple $t$ on $(\A,L_r)$ then the $(k+r)$-program can derive $P$ on $t$ on $(\A,L)$. The proof is by induction on the length of the derivation. Let $P(x) \; {:}{-} \; \phi$ be the last rule used in the derivation of the $k$-program and let $s$ be the satisfying assignment for the application of the rule. 
    We construct new rules. 
    If there is a variable $y$ such that $\phi$ contains the (EDB) conjunct $(L(s(y)))(y)$, then 
    \begin{itemize}
        \item if there is no IDB in $\phi$, then replace it by the (IDB) conjunct $(L(s(y)))(y)$
        \item if there is an (IDB) conjunct $Q(z)$ in $\phi$, then remove $(L(s(y)))(y)$ and replace $Q(z)$ by $(Q\times L(s(y)))(z,y)$.
    \end{itemize}
    The new rule is a rule of the $k$-program that can be applied with the satisfying assignment $s$ if the $(k+r)$-program can derive $Q\times L(x)$ on $(s(z),s(y))$. By induction the $(k+r)$-program can derive $Q$ on $s(z)$. 
    By definition there are unary relations with $L(s(y))=U_1\cap\dots\cap U_n$ such that the $r$-program can derive $U_i$ on $s(y)$ on $(\A,L)$ for all $i$.
    Let $R_1,\dots,R_m$ be a derivation of $U_1$ on $s(y)$ on $(\A,L)$. We construct a derivation $R_1',\dots,R_m'$ by replacing all (IDB) conjuncts $Q'(z')$ with $(Q\times Q')(z,z')$ and  adding the (IDB) conjunct $Q(z)$ to the rule $R_1$. Note that this construction adds arity of $Q$, so at most $k$, variables to the bodies. Hence, we obtain rules of the $(r+k)$-program that can after the derivation of $Q$ on $s(z)$ derive $Q\times U_1$ on $(s(z),s(y))$. Repeating this construction for the derivations of the other $U_i$ (using $(Q\times(U_1\cap\dots\cap U_{i-1})\times Q')(z,y,z')$ instead of $(Q\times Q')(z,z')$) the $(r+k)$-program can derive $Q\times (U_1\cap\dots\cap U_i)$ on $(s(z),s(y))$. Hence, it can derive $Q\times L(s(y))$ on $(s(z),s(y))$. 
    We can repeat this construction for every  variable $y'$ such that $\phi$ contains the (EDB) conjunct $(L(s(y')))(y')$ to get a derivation of the $(k+r)$-program that derives $P$ on $s(x)$.
\end{proof}

\begin{lemma}\label{lem:DatalogSimulateExpansion}
Let $(\A,L)$ be an instance. If $(\A,L)$ passes the $(k+\maxArity)$-test, then $(\expansion(\A),L)$ passes the $k$-test.
\end{lemma}
\begin{proof}
Let $P(x) \; {:}{-} \; \phi$ be a rule of the $k$-program, where $\phi$ contains an (EDB) conjunct $\pi_J(R)(y)$. Let $n$ be the arity of $R$ and let $y'$ be an $n$-tuple of variables such that $y'_J=y$ and the remaining variables of $y'$ are fresh. If $\phi$ contains no IDB, then let $\psi$ be the conjuncts of $\phi$ without $\pi_J(R)(y)$.
Consider the rules
\begin{align*}
P(x) &\; {:}{-} \; \pi_J(R)(y)\wedge \psi\\
    \pi_J(R)(y) &\; {:}{-} \; R(y')
\end{align*}
where $\pi_J(R)(y)$ in the first rule is an (IDB) conjunct. The body of the first rule has at most $k$ variables and the body of the second rule at most $\maxArity$.

If $\phi$ contains an (IDB) conjunct $Q(z)$, then let $\psi$ be the conjuncts of $\phi$ without $Q(z)$ and $\pi_J(R)(y)$. Consider the rules
\begin{align*}
P(x) &\; {:}{-} \; Q\times\pi_J(R)(z,y)\wedge \psi\\
    Q\times\pi_J(R)(z,y) &\; {:}{-} \; Q(z)\wedge R(y')
\end{align*}
The body of the first rule has at most $k$ variables and the body of the second rule has at most $k+\maxArity$ variables.

In both cases we can simulate the application of the rule $P(x) \; {:}{-} \; \phi$ with a rule of the $(k+\maxArity)$-program and a rule of the $k$-program that uses one less (EDB) conjunct of the form $\pi_J(R)$. Note that for all rules the reversed rule is also a rule of the corresponding program. 
The lemma follows by induction.
\end{proof}

\begin{lemma}\label{lem:DatalogSimulateUnary}
Let $(\A,L)$ be an instance, $x\in A$, $U\subseteq B$, and let $L'$ be obtained from $L$ by changing the value at $x$ to $L(x)\cap U$.
If $(\A,L')$ does not pass the $k$-test, then the $(k+1)$-program can derive $B\setminus U$ on $x$.    
\end{lemma}
\begin{proof}
    Consider a derivation using the rules $R_1,\dots,R_n$ of the $k$-program that derives the goal predicate on $(\A,L')$. Let $s_1,\dots,s_n$ be the satisfying assignments of the variables in the bodies of the rules. We construct rules $R'_1,\dots,R'_n$ as follows
    \begin{enumerate}
        \item replace each IDB $P(y)$ (in the head and the body) with the IDB $(U\rightarrow P)(z,y)$, where $z$ is a fresh variable and $U\rightarrow P$ is the relation $\{(a,t)\mid a\in B\text{ implies } t\in P\}$ and
        \item if there is a variable $y$ such that $s_i(y)=x$ and the body of $R_i$ contains the (EDB) conjunct $(L(x)\cap U)(y)$,  then replace $(L(x)\cap U)(y)$ with $(L(x))(y)$ and identify $y$ with the new variable $z$ introduced in the previous step.
    \end{enumerate}
    The new rules are rules of the $(k+1)$-program that derive $U\rightarrow G$ on $x$ on $(\A,L)$, where $G=\emptyset$ is the goal predicate. Note that $U\rightarrow G=B\setminus U$.
\end{proof}

\section{Further Results}

With a bit of work we can generalize Theorem~\ref{thm:main} to a much larger class of structures. 

\begin{corollary}\label{cor:main}
    Let $\B$ be a finite structure, such that any subset of $B$ of size at most three is pp-definable in $\B$, then the following are equivalent
    \begin{enumerate}
        \item $\csp(\B)$ is solved by a symmetric linear Datalog program,
        \item the polymorphisms of $\B$ contain a Hagemann-Mitschke chain and a 3-4WNU,
        \item $\B$ can neither pp-construct $\stCon$ nor $\lin p$ for any prime $p$,
        \item for no $a,b \in B$ distinct can $\B$ pp-define $O_{a,b}$ or $P_{a,b}$.
    \end{enumerate}
\end{corollary}
\begin{proof}
The implications $(1)\Rightarrow (3)\Leftrightarrow (2)\Rightarrow(4)$ are as in the proof of Theorem~\ref{thm:main}. 

$(4)\Rightarrow (1)$:
Let $\B_c$ be the expansion of $\B$ by all subsets of $B$.  If there are no $a,b \in B$ distinct such that can $\B_c$ pp-define $O_{a,b}$ or $P_{a,b}$, then Theorem~\ref{thm:main} implies $\csp(\B_c)$ is solved by a symmetric linear Datalog program. Hence, also $\csp(\B)$ is solved by a symmetric linear Datalog program. 
We have two cases to consider.
Assume there are $a,b\in B$ distinct, such that $\B_c$ can pp-define $O_{a,b}$. Let $\phi$ be a pp-formula with $\phi^{\B_c}=O_{a,b}$. Let $s_1,s_2,s_3$ be three satisfying assignments witnessing the three elements in $\phi^{\B_c}$. Construct $\psi$ from $\phi$ by removing all unary conjuncts and adding for each variable $x$ the conjunct $\{s_1(x),s_2(x),s_3(x)\}(x)$. Clearly, $\psi^{\B_c}=\phi^{\B_c}$. Since $\{s_1(x),s_2(x),s_3(x)\}$ is a unary relation that contains at most three elements it is a relation of $\B$. Therefore, $\psi$ can be interpreted in $\B$ and $\psi^{\B}=\psi^{\B_c}=\phi^{\B_c}=O_{a,b}$. A contradiction.

For the second case assume there are $a,b\in B$ distinct, such that $\B_c$ can pp-define $P_{a,b}$. We show that $P_{a,b}$ is preserved by all polymorphisms of $\B$, contradicting that $P_{a,b}$ is not pp-definable in $\B$. Let $f\colon\B^n\to \B$ be a polymorphism. For readability assume $a=0$ and $b=1$. Define $F\coloneqq\{0,1\}$.
Since $f$ preserves the two-element unary set $F$ we have $f(v)\in F$ for all $v\in F^n$. Therefore, $f$ gives us a map between the two $\F_2$-vector-spaces $\F_2^n$ and $\F_2$.
Every $v\in \F_2^n$ can be written as $1_S$, where $S\coloneqq\{i\in[n]\mid v_i=1\}$. Let $S,T\subseteq[n]$ with $S\cap T=\emptyset$. Define $h_{T,S}\colon B^3\to B,(x,y,z)\mapsto f(w)$, where $w_S=(x,\dots,x)$, $w_T=(y,\dots,y)$, and $w_{[n]\setminus(T\cup S)}=(z,\dots,z)$. The map $h_{T,S}$ is a minor of $f$. Hence, it is also a polymorphism of $\B$. Since $h$ is ternary and preserves all subsets of size at most three it already preserves all subsets. Therefore, $h$ is a polymorphism of $\B_c$ and preserves $P_{0,1}$. Hence,
\begin{align*}
    h(u)+h(v)+h((1,1,1)+u+v)=1\text{ for all $u,v\in \F_2^3$.}\\
\end{align*}
So $h((1,1,1)+v)=h(0,0,0)+h((1,1,1)+(0,0,0)+v)=1-h(v)=1+h(v)$ and
$h(u)+h(v)=1-h((1,1,1)+u+v)=1-1-h(u+v)=h(u+v)$. Therefore, $h$ is $\F_2$-linear on $F$ and
\begin{align*}
    f(1_{S\cup T})=h(1,1,0)=h(1,0,0)+h(0,1,0)=f(1_S)+f(1_T) \tag{$\ast$}
\end{align*}
Let $S,T\subseteq [n]$ and define $U\coloneqq S\cap T$, $S'\coloneqq S\setminus U$ and $T'\coloneqq T\setminus U$. Then
\begin{align*}
    f(1_{S} +1_{T})
    =f(1_{S'\cup T'})
    \stackrel\ast=f(1_{S'})+f(1_{T'})
    =f(1_{S'})+f(1_{T'})+2\cdot f(1_U)
    \stackrel\ast=f(1_S)+f(1_T).
\end{align*}
Hence, $f$ is $\mathbb F_2$-linear on $F$ and therefore it preserves $P_{0,1}$, as desired.
\end{proof}
Observe that proving $(2)\Rightarrow(1)$ directly is also possible: Hagemann-Mitschke chains use only ternary function symbols and Corollary~2.9 in \cite{JMMM} gives a minor condition that is (for finite structures) equivalent to 3-4WNU and uses only ternary function symbols. 
Hence, the polymorphisms of $\B_c$ contain a Hagemann-Mitschke chain and a 3-4WNU and (1) follows from Theorem~\ref{thm:main}.

\begin{DeclarationonGenerativeAI*}
ChatGPT6.0 Astra was used to solve the problem. The workflow for this article was as follows
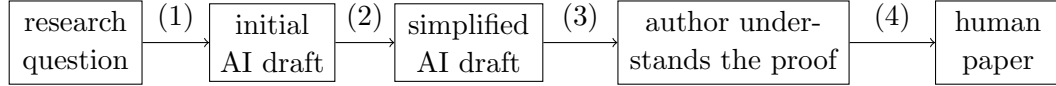
\begin{center}
    \begin{tikzpicture}
    \def\w{2cm}
        \node[rectangle,draw,text width=1.5cm,align=center] (0) at (0,0) {research question};
        \node[rectangle,draw,text width=1.4cm,align=center] (1) at (2.6,0) {initial AI draft};
        \node[rectangle,draw,text width=1.7cm,align=center] (2) at (5.2,0) {simplified AI draft};
        \node[rectangle,draw,text width=2.8cm,align=center] (3) at (8.7,0) {author understands the proof};
        \node[rectangle,draw,text width=1.4cm,align=center] (4) at (12.2,0) {human paper};
        \path[->]
        (0) edge node[above] {(1)} (1)
        (1) edge node[above] {(2)} (2)
        (2) edge node[above] {(3)} (3)
        (3) edge node[above] {(4)} (4)
        ;
    \end{tikzpicture}
\end{center}
Here step (1) is purely AI, step (2) was discussion of the author with AI, and step (3) and (4) are AI free. In particular, the document for the final paper was written by the author and is not an adaption of the AI draft. 
\end{DeclarationonGenerativeAI*}

\bibliographystyle{abbrv}
\bibliography{global}

\def\cprime{$'$} \def\cprime{$'$} \def\cprime{$'$}
\begin{thebibliography}{10}

\bibitem{BodDalJournal}
M.~Bodirsky and V.~Dalmau.
\newblock Datalog and constraint satisfaction with infinite templates.
\newblock {\em Journal on Computer and System Sciences}, 79:79--100, 2013.
\newblock A preliminary version appeared in the proceedings of the Symposium on
  Theoretical Aspects of Computer Science (STACS'05).

\bibitem{BulatovFVConjecture}
A.~A. Bulatov.
\newblock A dichotomy theorem for nonuniform {CSP}s.
\newblock In {\em 58th {IEEE} Annual Symposium on Foundations of Computer
  Science, {FOCS} 2017, {B}erkeley, {CA}, {USA}, {O}ctober 15-17}, pages
  319--330, 2017.

\bibitem{LinearDatalog}
V.~Dalmau.
\newblock Linear {D}atalog and bounded path duality of relational structures.
\newblock {\em Logical Methods in Computer Science}, 1(1), 2005.

\bibitem{DalmauLICS15}
V.~Dalmau, L.~Egri, P.~Hell, B.~Larose, and A.~Rafiey.
\newblock Descriptive complexity of list h-coloring problems in logspace: {A}
  refined dichotomy.
\newblock In {\em 30th Annual {ACM/IEEE} Symposium on Logic in Computer
  Science, {LICS} 2015, Kyoto, Japan, July 6-10, 2015}, pages 487--498. {IEEE}
  Computer Society, 2015.

\bibitem{EgriLaroseTessonLogspace}
L.~Egri, B.~Larose, and P.~Tesson.
\newblock Symmetric {D}atalog and constraint satisfaction problems in logspace.
\newblock In {\em Proceedings of the Symposium on Logic in Computer Science
  ({LICS})}, pages 193--202, 2007.

\bibitem{EgriLT08}
L.~Egri, B.~Larose, and P.~Tesson.
\newblock Directed st-connectivity is not expressible in symmetric {D}atalog.
\newblock In {\em Automata, Languages and Programming, 35th International
  Colloquium, {ICALP} 2008, Reykjavik, Iceland, July 7-11, 2008, Proceedings,
  Part {II} - Track {B:} Logic, Semantics, and Theory of Programming {\&} Track
  {C:} Security and Cryptography Foundations}, pages 172--183, 2008.

\bibitem{FederVardi}
T.~Feder and M.~Y. Vardi.
\newblock The computational structure of monotone monadic {SNP} and constraint
  satisfaction: {a} study through {D}atalog and group theory.
\newblock {\em {SIAM} Journal on Computing}, 28:57--104, 1999.

\bibitem{JMMM}
J.~Jovanovi\'c, P.~Markovi\'c, R.~McKenzie, and M.~Moore.
\newblock Optimal strong {M}al'cev conditions for congruence
  meet-semidistributivity in locally finite varieties.
\newblock {\em Algebra Universalis}, 76:305--325, 2016.

\bibitem{Kazda-n-permute}
A.~Kazda.
\newblock {$n$-permutability and linear {D}atalog implies symmetric {D}atalog}.
\newblock {\em {Logical Methods in Computer Science}}, {Volume 14, Issue 2},
  Apr. 2018.

\bibitem{StarkeDiss}
F.~Starke.
\newblock Digraphs modulo primitive positive constructability.
\newblock Preprint, 2024.
\newblock PhD dissertation, Institute of Algebra, TU Dresden.

\bibitem{ZhukFVConjecture}
D.~N. Zhuk.
\newblock A proof of {CSP} dichotomy conjecture.
\newblock In {\em 58th {IEEE} Annual Symposium on Foundations of Computer
  Science, {FOCS} 2017, {B}erkeley, {CA}, {USA}, {O}ctober 15-17}, pages
  331--342, 2017.
\newblock \url{https://arxiv.org/abs/1704.01914.}

\end{thebibliography}

\end{document}